\documentclass[12pt]{article}

\usepackage{changepage}
\usepackage{tabto}

\usepackage{enumitem}
\usepackage[mathscr]{euscript}

\usepackage{geometry}
\usepackage{amsmath}  \usepackage{amsthm} \usepackage{amssymb}

\theoremstyle{definition}

\swapnumbers                
\newtheorem{thm}{Theorem}   
\newtheorem{cor}[thm]{Corollary}
\newtheorem{lem}[thm]{Lemma}
\newtheorem{prop}[thm]{Proposition}
\newtheorem{dfn}[thm]{Definition}

\newtheorem{aside}[thm]{Remark}

\newtheorem{oques}[thm]{Open Questions}

\usepackage[font=footnotesize]{caption}
\usepackage[colorlinks=true,
            linkcolor=red,
            urlcolor=blue,
            citecolor=blue]{hyperref}     

\usepackage{titlesec}
\titleformat{\section}
  {\titlerule[0.8pt]\normalfont\Large\bfseries}{\thesection}{1em}{}

\usepackage{relsize,etoolbox} 
\usepackage{quoting}             
\quotingsetup{vskip=0pt,rightmargin=0pt}
\AtBeginEnvironment{quoting}{\smaller}

\newcommand{\spz}{\hspace*{.4cm}}
\newcommand{\spa}{\hspace*{1cm}}

\newcommand{\spb}{\hspace*{2cm}}

\newcommand{\spc}{\hspace*{3cm}}
\newcommand{\spd}{\hspace*{4cm}}

\newcommand{\vsp}{\vspace*{1cm}}

\newcommand{\pp}{\mathscr{P}}

\newcommand{\zz}{\mathbb{Z}}

\newcommand{\qq}{\mathbb{Q}}

\newcommand{\cc}{\mathbb{C}}

\newcommand{\bp}{\mathbb{BP}}

\newcommand{\sdot}{\!\cdot\!}

\newcommand{\lcm}{\text{lcm}}

\newcommand{\bin}{binomid}  
\newcommand{\bim}{binomid}  
\newcommand{\Bim}{Binomid}  
\newcommand{\bino}{nomid}  

\newcommand{\lu}{U}      

\newcommand{\gc}{GCD}

\newcommand{\ds}{\displaystyle}
\newcommand{\sz}{\scriptsize}

\newcommand{\bep}{\boldsymbol{!}}      
\newcommand{\la}{\langle}
\newcommand{\ra}{\rangle\bep}           

\newcommand{\abin}[2]{ \ensuremath{ \left[\! \begin{array}{c} #1 \\ #2 \end{array}\! \right]  } }

\newcommand{\fbin}[2]{ \ensuremath{ \left[ \begin{smallmatrix} #1 \\ \rule{0mm}{3mm}#2 \end{smallmatrix} \right] }}
\newcommand{\dgcd}{ComboSum}

\title{\bf \large Divisibility Properties for Integer Sequences} 
\author{\small  Daniel B. Shapiro\\[-5pt] \small   Department of Mathematics \\[-5pt] \small Ohio State University\\[-5pt] \small  Columbus, OH 43210}
\date{}

\begin{document}

\begin{center}
{\bf \large Divisibility Properties for Integer Sequences}

\vsp
\textbf{ABSTRACT}
\end{center}
\vspace*{-3mm}
A sequence of nonzero integers $f = (f_1, f_2, \dots)$ is \emph{\bin} if  every $f$-\bino\ coefficient $\fbin{n}{k}_f$ is an integer.  Those terms are the generalized binomial coefficients:

\spd $\displaystyle \abin{n}{k}_{f}  \;=\; \frac{f_nf_{n-1}\cdots f_{n-k+1}}{\!\!\!\!\!f_k\,f_{k-1}\phantom{i}\cdots\phantom{i} f_1}. $

Let $\Delta(f)$ be the infinite triangle with those numbers as entries. When $I = (1, 2, 3, \dots)$, $\Delta(I)$ is Pascal's Triangle and we see that  $I$ is \bin.

Surprisingly, every row and column of Pascal's Triangle is \bin.  If fact, for any  $f$, each row and column of $\Delta(f)$ generates its own triangle. Those triangles fit together to form the ``\Bim\ Pyramid'' $\bp(f)$.  Sequence $f$ is \emph{\bin\ at every level} if all entries of $\bp(f)$ are integers.

We prove that several familiar sequences are \bin\ at every level.  For instance, every sequence $L$ satisfying a linear recurrence of order 2 has that property provided $L(0) = 0$.  The sequences $I$, the Fibonacci numbers, and $(2^n - 1)_{n \ge 1}$ provide examples.

\newpage

\maketitle

\section{ Introduction.}

In this paper, we consider sequences $f = (f_n)_{n \ge 1} = (f_1, f_2, \dots)$ where every $f_n$ is a nonzero integer.  We sometimes write $f(n)$ in place of $f_n$.

\begin{dfn} \label{fbin_coeff}  
For integers $n, k$ with $0 \le k \le n$, the  \emph{$f$-\bino}

coefficient is: \\[-3pt]

\spd\spa $\displaystyle \abin{n}{k}_{f}  \;=\; \frac{f_nf_{n-1}\cdots f_{n-k+1}}{\!\!\!\!\!\!f_k\,f_{k-1}\phantom{i}\cdots\phantom{i} f_1}$. 

Defining \emph{$f$-factorials} as: \ 
$\la n \ra _f =  f_nf_{n-1}\cdots f_2f_1$ we find that
\smallskip

\spd\spa   $\abin{n}{k}_{f} \;=\; \dfrac{ \la n \ra_{f} }{ \la k \ra_f\!\cdot\! \la n-k \ra_f }$
\end{dfn}

Set $\la 0 \ra = 1$, and note that $\fbin{n}{0}_f\!  = \fbin{n}{n}_f\! =  1$.  Also, $\fbin{n}{k}_f$ is undefined when $k > n$ and when $n$ or  $k$ is negative. (Some authors set  $\fbin{n}{k}_f = 0$ when $k > n$.)
That factorial formula helps explain the symmetry:

\spd\spa    $\abin{n}{k}_f\! \;=\;  \abin{n}{n-k}_f\!$, \ whenever $0 \le k \le n$.   \\

These generalized binomial coefficients have been considered by several authors in various contexts.  Ward \cite{mW36} mentioned them in 1936 and wrote several subsequent papers discussing their properties and applications. For instance, in \cite{mW36a} he developed a theory of calculus that includes analogues of power series.   Gould  \cite{hG69} describes the early history of $f$-\bino\ coefficients (that he calls Fonten\'{e}-Ward coefficients), and mentions some of their properties. Knuth and Wilf \cite{KW89} provide a few more early references, and Ballot \cite{cB14} investigates related ideas.  

\vfill

\pagebreak

\begin{dfn} \label{fbin_triangle}  \label{bin_def}
The \textit{\bim\ triangle} $\Delta(f)$ is the triangular array of all the $f$-\bino\ coefficients $\fbin{n}{k}_f$ for $0 \le k \le n$.  
Sequence  $f$ is \emph{\bin} if every entry of $\Delta(f)$ is an integer.  \ Then $f$ is \bin\ if and only if: 

\spb  \fbox{\strut\ $f_1f_2 \cdots f_k$ \;divides\; $f_{m+1}f_{m+2} \cdots f_{m+k}$, } \  for every $m$ and $k$ in $\zz^+$
\end{dfn}

\vspace*{-5pt}
If  $f_1 = 1$ then $f$ appears as Column 1 of its triangle $\Delta(f)$, since $\fbin{n}{1}_f = f_n$.  \\
The classic ``Pascal's Triangle'' of binomial coefficients is $\Delta(I)$ where  $I = (n) = (1, 2, 3, 4, \dots)$. Some 
history of this numerical triangle is mentioned in the OEIS \cite{OEIS} A007318 and in  \url{https://en.wikipedia.org/wiki/Pascals_triangle}.

\medskip

It is convenient to allow finite sequences.  Suppose there is  $N \ge 1$ such that 

\spb $f_n \ne 0$ for $1 \le n \le N$ \ and \ $f_n$ is not defined for $n > N$.

When $0 \le k \le n \le N$, define $\fbin{n}{k}_f $ as before, leaving $\fbin{n}{k}_f $ undefined when $n > N$. Extend Definition \ref{bin_def} similarly: $\Delta(f)$ is a finite triangular array with $N+1$ entries along each edge.

\medskip

To verify notations, we display  the first few values of Pascal's Triangle $\Delta(I)$ below, with index $n \ge 0$ vertically on the left, and index  $k \ge 0$ across the top.

\spc {\small 
\begin{tabular}{c|cccccccccc}
       & 0   & 1  & 2   &  3    &   4     &    5    &    6   &   7  &   8   \\    
\hline
   0  & 1 \\ 
   1  & 1  &  1 \\
   2  & 1  &  2  & 1 \\
   3  & 1  &  3  &  3   & 1  \\
   4  &  1 &  4  &  6   &   4   &   1 \\
   5  &  1 &  5  &  10 &  10  &    5    &   1  \\
   6  &  1 &  6  & 15  &  20  &   15   &    6  &    1  \\
   7  &  1 &  7  &  21 &  35  &   35   &  21  &    7   &   1   \\
   8  &  1 &  8  &  28 &  56  &   70   &  56  &   28  &   8  &  1 \\
   \vdots &   \vdots &  \vdots &  \vdots &  \vdots &  \vdots &  \vdots &   \vdots &   \vdots &    \vdots &   $ \ddots $ 
 
\end{tabular}
}\\[3pt]
\spd  {\footnotesize Classic Pascal triangle $\Delta(I)$ for $I(n) = n$.  }

There are several ways to prove that every $\binom{n}{k}$ is an integer.  For instance, $\binom{n}{k}$ is the number of elements in some set, so it is a non-negative integer. An inductive proof uses the formula $\binom{n+1}{k+1} = \binom{n}{k+1} + \binom{n}{k}$.   A third proof arises from prime factorizations.

The following Lemma can be used to generate examples of \bin\ sequences.

\begin{lem} \label{bin-basic} Suppose $f, g$ are integer sequences.  
\begin{enumerate}[label=(\arabic*),topsep=0pt]

\item \label{prods} Define sequence  $fg$ by: $(fg)_n= f_ng_n$.  
Then                $\abin{n}{k}_{fg} = \abin{n}{k}_f \!\!\!\cdot \abin{n}{k}_g$.

If $f$ and $g$ are \bin, then $fg$ is \bin.  \\
If $c$ is nonzero in $\zz$  then: \ $cf$ is \bin\ \ if and only if \ $f$ is \bin.

\item \label{div_chain} Suppose $f$ is a \textit{divisor-chain}: $f_n \mid f_{n+1}$ for every $n$. Equivalently: $f_n = \la n \ra_a$ for some integer sequence $a$. \\ 
\spa If $f$ is a divisor-chain then $f$ is \bin.  \\
In particular, $(c^n) = (c, c^2, c^3, \dots)$ and  $n! = (1, 2, 6, 24, 120, \dots)$ are \bin.

\item \label{homom}
Define an integer sequence $\psi$ to be \textit{homomorphic},\footnote{Other names include ``totally multiplicative'' and ``strongly multiplicative.''\medskip } if   \ $\psi(mn) = \psi(m)\psi(n)$ \ for every $m, n$.   If $f$ is \bin\ then $\psi \circ f$ is also \bin.

\item \label{chain} If $f = (f_1, f_2, f_3, \dots)$ is \bin\, then so are the sequences $(1, f_1, f_2, f_3, \dots)$ and $(1, f_1, 1, f_2, 1, f_3, \dots )$ and $(f_1, f_1, f_2, f_2, f_3, f_3, \dots )$.

\end{enumerate}
\end{lem}

Proofs are left to the reader.   \qed   

\bigskip

The ``triangular numbers''  $T(n) = \frac{n(n+1)}{2}$ are the entries in Column 2 of Pascal's triangle. 
 Here are the first few rows of $\Delta(T)$.  

\spc {\small 
\begin{tabular}{c|ccccccccc}
       & 0   & 1  & 2   &  3    &   4     &    5    &    6   &   7    \\    
\hline
   0  & 1 \\ 
   1  & 1  &  1 \\
   2  & 1  &  3  & 1 \\
   3  & 1  &  6  &  6   & 1  \\
   4  &  1 &  10  &  20   &   10   &   1 \\
   5  &  1 &  15  &  50   &  50  &    15    &   1  \\
   6  &  1 &  21  & 105     &  175  &   105   &    21  &    1  \\
   7  &  1 &  28  &  196  &  490  &   490   &  196  &    28   &   1   \\
   \vdots &   \vdots &  \vdots &  \vdots &  \vdots &  \vdots &  \vdots &   \vdots & \vdots &  $ \ddots $ 
\end{tabular}
}\\[3pt]
\spd  {\footnotesize Triangle for $T(n) = C_2(n) = n(n+1)/2$.  }\\[-6pt]

Similarly, the third Pascal Column $C_3 = (1, 4, 10, 20, 35, \dots )$ generates the \bim\ triangle:

\spc {\small 
\begin{tabular}{c|lllllllll}
       & 0   & 1  & 2   &  3    &   4     &    5    &    6   &   7    \\    
\hline
   0  &  1 \\ 
   1  &  1  &  1 \\
   2  &  1  &  4    & 1 \\
   3  &  1  &  10  &  10    & 1  \\
   4  &  1  &  20  &  50    &    20    &   1 \\
   5  &  1  &  35  &  175   &   175   &    35    &   1  \\
   6  &  1  &  56  &  490   &   980   &   980    &    56    &    1  \\
   7  &  1  &  84  &  1176  &  4116  &  4116   &  1176  &    84   &   1   \\
   \vdots  &  \vdots &  \vdots &  \vdots &  \vdots &  \vdots &  \vdots &   \vdots & \vdots &  $ \ddots $ 
 
\end{tabular}
}\\[3pt]
\spd  {\footnotesize Triangle for $C_3(n) = n(n+1)(n+2)/6$.  }

Those integer values lead us to suspect that every Pascal column $C_m$ is \bin.\

 That result has been proved by several authors.  It also follows from our Theorem \ref{DP_Bin} and Lemma \ref{DP_basics}(3) below.

The OEIS  \cite{OEIS} webpage for sequence A342889 provides many references related to the triangles $\Delta(C_m)$. For instance, $\Delta(T) = \Delta(C_2)$ is the triangle of Narayana numbers. 
The ``generalized binomial coefficients''\footnote{Some authors use that term for different numerical triangles. For instance, the numbers $\binom{n}{m}_s$ in \cite{bB93} p. 15 are not the same as our \bin\ coefficients.  Similarly the ``Pascal pyramid'' built from trinomial coefficients (as in  \cite{bB93} p. 45) is not one of the \Bim\ Pyramids mentioned below.}   
$(n,k)_m$ mentioned on the OEIS pages are our $\fbin{n}{k}_{ C_{m+1} }$ built from Pascal columns.  

Certain determinants of binomial coefficients are related to sequence $C_m$:  \smallskip


\spb $  \det \Big[ \binom{ \rule[-1.5mm]{0mm}{3mm}  n+i}{m+j}  \Big]_{ i,j =0}^{k-1} \ = \   \frac{\rule[-2mm]{0mm}{3mm}  \binom{n+k-1}{m} \ \cdots  \ \binom{n}{m} }{\rule[1mm]{0mm}{3mm}  \binom{m+k-1}{m} \ \cdots \ \binom{m}{m}  }  \ = \ \fbin{n-m+k}{k}_{C_m}$.

\smallskip
Consequently, every $C_m$ is a \bin\ sequence.  That determinant formula is mentioned in \cite{bB93} p. 164, referring to \cite{eN01} p. 257.\footnote{Netto states that this determinant formula is due to an 1865 work of v.~Zeipel.}

Gessel and Viennot  \cite{GV85} found combinatorial interpretations for such Pascal determinants.  In a recent exposition, Cigler \cite{jC21} derives determinant expressions for entries of  $\Delta(C_m)$, arrays that he calls ``Hoggatt Triangles'' following \cite{FA88}.

\bigskip \medskip

What about the Pascal rows?  
 Here are some triangles for rows $R_m(n) = \binom{m}{n-1}$.  

{\small 
\begin{tabular}{c|ccccc}
       & 0   & 1  & 2   &  3       \\    
\hline
   0  & 1 \\ 
   1  & 1  &  1 \\
   2  & 1  &  2  & 1 \\
   3  & 1  &  1  &  1   & 1  \\
\end{tabular}
}  \spd
{\small 
\begin{tabular}{c|ccccccc}
       & 0   & 1  & 2   &  3    &  4   \\    
\hline
   0  & 1 \\ 
   1  & 1  &  1 \\
   2  & 1  &  3  & 1 \\
   3  & 1  &  3  &  3   & 1  \\
   4  & 1  &  1  &  1   & 1  &  1  \\
\end{tabular}
}\\[3pt]
\spz  {\footnotesize Triangle for $R_2(n) = \binom{2}{n-1}$.  }  \spc  {\footnotesize Triangle for $R_3(n) = \binom{3}{n-1}$.  }\\

\vspace*{4mm}

{\small 
\begin{tabular}{c|cccccccccccc}
       & 0   & 1  & 2   &  3 &  4  &  5    \\    
\hline
   0  & 1 \\ 
   1  & 1  &  1 \\
   2  & 1  &  4  & 1 \\
   3  & 1  &  6  &  6   & 1  \\
   4  & 1  &  4  &  6   & 4  &  1  \\
   5  & 1  &  1  &  1   & 1  &  1  &  1 \\
\end{tabular}
}  \spc 
{\small 
\begin{tabular}{c|cccccccccccc}
       & 0   & 1  & 2   &  3    &  4  &  5  &  6  \\    
\hline
   0  & 1 \\ 
   1  & 1  &  1 \\
   2  & 1  &  5  & 1 \\
   3  & 1  &  10  &  10   & 1  \\
   4  & 1  &  10  &  20   & 10 &  1 \\
   5  & 1  &  5  &  10   & 10  &  5  &  1  \\
   6  & 1  &  1  &  1   & 1  &  1  &  1  &  1 \\
\end{tabular}
}\\[3pt]
\spz  {\footnotesize Triangle for $R_4(n) = \binom{4}{n-1}$.  }  \spc  {\footnotesize Triangle for $R_5(n) = \binom{5}{n-1}$.  }

\vfill
\pagebreak

{\small 
\begin{tabular}{c|cccccccccccc}
       & 0   & 1  & 2   &  3 &  4  &  5  & 6  &  7    \\    
\hline
   0  & 1 \\ 
   1  & 1  &  1 \\
   2  & 1  &  6  & 1 \\
   3  & 1  &  15  &  15   & 1  \\
   4  & 1  &  20  &  50   & 20  &  1  \\
   5  & 1  &  15  &  50   & 50  &  15  &  1 \\
   6  &  1 &   6   &  15  &  20  &  15  &   6  &   1 \\
   7  & 1  &   1   &   1   &   1  &  1  &  1  &  1  &  1
\end{tabular}
}  \spa  \vspace*{1mm}
{\small 
\begin{tabular}{c|cccccccccccccc}
       & 0    & 1  & 2  &  3    &  4    &   5   &  6   &   7   &  8  \\    
\hline
   0  & 1 \\ 
   1  & 1  &   1 \\
   2  & 1  &   7    &    1 \\
   3  & 1  &  21   &   21   & 1  \\
   4  & 1  &  35   &  105   &  35   &    1 \\
   5  & 1  &  35   &  175   & 175  &  35  &   1  \\
   6  & 1  &  21   &  105   & 175  & 105 &  21  &  1 \\
   7  & 1  &   7    &    21   &  35  &  35  &  21  &  7  &  1  \\
   8  & 1  &   1    &     1    &   1   &   1   &   1  &  1  &  1  &  1
\end{tabular}
}\\[3pt]
\spz  {\footnotesize Triangle for $R_6(n) = \binom{6}{n-1}$.  }  \spc\quad  {\footnotesize Triangle for $R_7(n) = \binom{7}{n-1}$.  }

\bigskip

Define the \textbf{Binomial Pyramid} to be the numerical pyramid built by stacking the triangles built from Pascal Rows.  That infinite pyramid has three faces with all outer entries equal to 1.  The horizontal slice at depth $m$ is $\Delta(R_m)$, the triangle built from Pascal Row $m$. Our numerical examples indicate that all entries are integers and the pyramid has three-fold rotational symmetry.  

Removing one face (of all ones) from that pyramid exposes an infinite triangular face that is the original Pascal triangle.  This is seen in the triangles $\Delta(R_m)$ displayed above: each Column 1 is a row of Pascal's triangle (by construction).

Listing Column 2 for those triangles displayed above yields\\
\spb  (1, 1), \, (1, 3, 1), \; (1, 6, 6, 1), \; (1, 10, 20, 10, 1), \; etc.  

Those are exactly the rows of the \bim\ triangle for $T = C_2 = (1, 3, 6, 10, \dots)$ \linebreak 
 displayed earlier!  That is, removing two face-layers of the \Bim\ Pyramid exposes a triangular face that equals $\Delta(C_2)$, built from Pascal Column 2.  This pattern continues:  Each triangle $\Delta(C_m)$ appears as a slice of that \Bim\ Pyramid.  We generalize those assertions here, and outline proofs in the next section.

\begin{dfn} \label{pyramid} For a sequence $f$, the \textbf{\Bim\ Pyramid} 
$\bp(f)$  is made by stacking the \bin\ triangles constructed from the rows of triangle $\Delta(f)$.

Sequence $f$ is\textit{ \bin\ at level $c$} if column $c$ of $\Delta(f)$ is a \bin\ sequence. 
\end{dfn}

By definition, $f$ is \textit{\bin\ at every level} if each column of $\Delta(f)$ is \bin.  
By Corollary \ref{ColRowBin} below, this is equivalent to saying: Pyramid $\bp(f)$ has all integers entries.

\medskip
We mentioned that $T = (1, 3, 6, 10, 15, \dots)$ is \bin\ at level 1.  Note that $T$ is not \bin\ at level 2 since Column 2 of $\Delta(T)$  is $g = (1, 6, 20, 50, 105, \dots)$ and $\fbin{4}{2}_g$ is not an integer.  

Some sequences are \bin\ at level 2 but not at level 1.  For instance, let $f$ be the eventually constant sequence $f = (2^{a_n})$ where: $a = (0, 2, 4, 1, 3, 1, 4, 4, 4, \dots )$. \linebreak
It is not \bin\ because $\fbin{6}{3}_f = \frac{1}{2}$.  Check that Column 2 of $\Delta(f)$ is $(2^{b_n})$ where 
$b = (0, 4, 3, 2, 2, 3, 6, 6, 6, \dots)$. By checking several cases, we can show that this column is \bin. 

\bigskip
Here are some examples that are fairly easy to verify.

\begin{prop} \label{every_level_1} Continue with notations in Lemma \ref{bin-basic}:
\begin{enumerate}[label= (\arabic*),topsep=0pt,itemsep=0pt,partopsep=4pt, parsep=0pt,before=\setlength{\baselineskip}{0mm},itemsep=3pt] \label{DivChain_Bin} 

\item  Every divisor-chain is \bin\ at every level.

\item Suppose the integer sequence $\psi$ is homomorphic. If $f$ is \bin\ at level $c$, then so is $\psi \circ f$.

\end{enumerate}
\end{prop}

\begin{proof}

Statement (2) follows from Lemma \ref{bin-basic}(3).  (1) Suppose $f$ is a divisor-chain.  Let $C_j$ be column $j$ of $\Delta(f)$, so that $C_j(n) = \fbin{n + j - 1}{j}_f$.  Since $\fbin{d}{\rule{0mm}{4mm}j}_f = \frac{f_{d}}{f_{d-j}}\fbin{\rule{0mm}{2mm}d-1}{\rule{0mm}{3.5mm}j}_f$ and  $\frac{f_d}{f_{d-j}}$ is an integer whenever $0 \le j < d$, we conclude that $C_j$ is a divisor-chain. Therefore $C_j$ is \bin\ by Lemma \ref{bin-basic}(2). 
\end{proof}

\medskip 
As another motivating example, let $G_2(n) = 2^n - 1 = (1, 3, 7, 15, \dots)$.   Here are the first few rows of its triangle. Those integer entries indicate that $G_2$ is \bin.

\spc {\small 
\begin{tabular}{c|ccccccccc}
       & 0   & 1  & 2   &  3    &   4     &    5    &    6       \\    
\hline
   0  & 1 \\ 
   1  & 1  &  1 \\
   2  & 1  &  3  & 1 \\
   3  & 1  &  7  &  7   & 1  \\
   4  &  1 &  15  &  35   &   15   &   1 \\
   5  &  1 &  31  &  155 &  155  &    31    &   1  \\
   6  &  1 &  63  & 651  &  1395  &   651   &    63  &    1  \\
   \vdots &   \vdots &  \vdots &  \vdots &  \vdots &  \vdots &  \vdots &   \vdots & $ \ddots $ 
\end{tabular}}\\[4pt]
\spd  {\footnotesize Triangle for $G_2(n) = 2^n-1$.  }

\bigskip

Let  $G_q(n) = \frac{q^n - 1}{q - 1} = 1 + q + \cdots + q^{n-1}$.  
The entries  $\fbin{n}{k}_{G_q}$ in triangle $\Delta(G_q)$ are often called ``$q$-nomial'' (or Gaussian) coefficients and have appeared in many articles since Gauss  \cite{cG08} introduced them in 1808. E.g. see \cite{KW89}.  
 ``Fibonomial'' coefficients are the entries of  $\Delta(F)$  where $F$ is the Fibonacci sequence.   
As Lucas \cite{eL78} and Carmichael \cite{rC13} pointed out long ago,  $G_q$ and $F$ are examples of \emph{Lucas sequences}.\footnote{Ballot \cite{cB14} refers to entries of $\Delta(L)$ as ``Lucanomial'' coefficients. }  Those are integer sequences $L$ that satisfy a linear recurrence of order 2 and have $L(0) = 0$. Such $L$ is a constant multiple of the sequence $U$ in Definition \ref{lucas}.  Carmichael's Theorem VII in \cite{rC13} implies that every Lucas sequence is \bin.  \\
\spz Our Theorem \ref{DP_Bin} and Lemma \ref{lucas} below imply that every Lucas sequence is \bin\ at every level.

\bigskip

\begin{aside} \label{log-bin} \!
The definition of \bin\ sequences can be restated in additive form.   
Let $v_p(n)$ be the exponent of the prime $p$ in $n$. That is,  $n = \prod\limits_{p} p^{v_p(n)}$
.  
\begin{enumerate}[label=(\arabic*),topsep=-2pt]  
\item Suppose $(a) = (a_1, a_2, \dots)$ is a never-zero sequence of integers.  Then  $(a)$ is \bin\ \;if and only if\;  $(p^{v_p(a_n)})$ is \bin\ for every prime $p$. 

\item For an integer sequence $b = (b_1, b_2, \dots)$, define $s_b(n) = b_1 + \cdots + b_n$. 
Suppose $c > 1$ is an integer. Then:

\spz  $(c^{b_n})$ is \bin\ \;$\iff$ \ 
  $s_b(m) + s_b(n) \le s_b(m+n)$ \   for every $m, n \in \zz^+$.   
\end{enumerate}
\smallskip

Each property defined below has an additive version.  But those reformulations do not seem to provide significantly better proofs of the Theorems. 

\end{aside}

\newpage
\section{ Binomid Pyramids.}

For a sequence $f$, the \Bim\ Pyramid $\bp(f)$ is formed by stacking the \bim\ triangles of the row-sequences of the triangle  $\Delta(f)$. In this section we verify that the \bim\ triangles for the column-sequences of $\Delta(f)$ appear by slicing that Pyramid along planes parallel to a face.

\medskip
Our sequences start with index 1. We choose notations so that the row and column sequences begin with index 1.  We often restrict attention to sequences $f$ with $f_1 = 1$.  

\begin{dfn} \label{row-col-def}
If $f$ is a sequence with $f_1 = 1$, define 
the Row and Column sequences of its triangle $\Delta(f)$ by:

\spc   $R_m(N) = \fbin{\rule{0mm}{2.5mm} m}{\rule{0mm}{4mm} N-1}_f$, \spz and  \spz  $C_j(N) = \fbin{N + j - 1}{j}_f$. 
\end{dfn}

\medskip

Row sequence $R_m$ has only $m+1$ entries:   

\spz $R_m = \Big( \fbin{ \rule[-0.5mm]{0mm}{3mm}m}{0}_f, \fbin{ \rule[-0.5mm]{0mm}{3mm}m}{1}_f, \fbin{\rule[-0.5mm]{0mm}{3mm}m}{2}_f, \dots , \fbin{ \rule[-0.5mm]{0mm}{3mm}m}{m}_f  \Big)$, with  $R_m(N)$ undefined for  $N > m+1$. 

For example: \\
\spz $R_0 = (1), \; R_1 = (1,1)$, \; $R_2 = (1, f_2, 1)$, \; \dots \; ,
 $R_m = (1, f_m,  \frac{f_mf_{m-1}}{f_2}, \,\dots\, , f_m, 1)$.\\

Column sequences are infinite, with entries 

\spa $C_j = \Big( \fbin{\rule[-0.5mm]{0mm}{3mm}j}{j}_f, \fbin{j+1}{j}_f, \fbin{j+2}{j}_f, \ \dots \Big) 
   = (1, f_{j+1}, \ \frac{f_{j+2}f_{j+1}}{f_2}, \ \dots )$. 

Note that $C_0 = (1, 1, 1, \dots)$ and $C_1 = (1, f_2, f_3, \dots ) = f$.

\medskip

\Bim\ triangles for the first few Pascal rows have 3-fold rotational symmetry.  Proposition \ref{symmetry} shows that this follows from the left/right symmetry of each row. 
 In our triangle displays,  this says that each column matches a downward diagonal. 

\medskip

For example, sequence $f = (1, a, b, b, a, 1)$ appears 3 times in its triangle $\Delta(f)$:

\spc {\small 
\begin{tabular}{c|cccccccccccc}
        &  {\sz 0}   & {\sz 1}  & {\sz 2}   &  {\sz 3}    &  {\sz 4}  &  {\sz 5}  &  {\sz 6}  \\    
\hline
   {\sz 0}  & 1 \\ 
   {\sz 1}  & 1  &  1 \\
   {\sz 2}  & 1  &  $a$  & 1 \\
   {\sz 3}  & 1  &  $b$  &   $b$         & 1  \\
   {\sz 4}  & 1  &  $b$  &  {\scriptsize $b^2\!/a$}   & $b$ &  1 \\
   {\sz 5}  & 1  &  $a$  &   $b$         & $b$  &  $a$  &  1  \\
   {\sz 6}  & 1  &   1    &     1           & 1  &  1  &  1  &  1 \\
\end{tabular}
}

\bigskip
\begin{lem}\label{symmetry}
Suppose $f= (f_1, f_2, \dots, f_n)$ is a symmetric list of $n$ terms. 
That is:  $f_k = f_{n+1 - k}$. Then the triangle  $\Delta(f)$ has 3-fold rotational symmetry.  
\end{lem}

\begin{proof}  We need to prove: \    Column $c$ = Row ${n-c}$, \;whenever $0 \le c \le n$. 

The $(k+1)^{\text{st}}$ entries of $C_c$ and $R_{n-c}$ are 

\spa  $\abin{\!c+ k\!}{k}_f  = \ds\frac{f_{c+k}f_{c+k-1}\cdots f_{c+1}}{\la k \ra_f}$  \; and \;  $ \abin{\!\rule[-0.5mm]{0mm}{3mm}n - c\!}{k}_f = \ds\frac{f_{n-c}f_{n-c-1}\cdots f_{n-c-k+1}}{\la k \ra_f}$

\smallskip
By symmetry of $f$, those numerators are equal term by term:

\spb  $f_{c+k} =  f_{n-c-k+1}, \quad f_{c+k-1} = f_{n-c-k+2}, \quad \dots\ ,   \quad f_{c+1} = f_{n-c}$. 
\end{proof}

\medskip
For the finite sequence above, the symmetry shows that $R_n = (1, 1, \dots, 1)$, \;and  every later row is undefined.  

\bigskip
The \Bim\ Pyramid $\bp(f)$ in Definition \ref{pyramid} is built by stacking the triangles $\Delta(R_m)$.  To show that the triangle $\Delta(C_j)$  appears as a slice of this pyramid, we check that each row of $\Delta(C_j)$ equals a corresponding row and column of the horizontal slice $\Delta(R_m)$.  Numerical observations indicate that

\spa Row $n$ of $\Delta(C_2)$ = Column 2 of $\Delta(R_{n+1})$ = Row $n$ of $\Delta(R_{n+1})$.

\spa Row $n$ of $\Delta(C_3)$ = Column 3 of $\Delta(R_{n+2})$ = Row $n$ of $\Delta(R_{n+2})$.

The general pattern provides the next result: 

\begin{prop} \label{ColRow}
For a sequence $f$ and every $n$ and $m$:

\spa Row $n$ of $\Delta(C_m)$ = Column $m$ of $\Delta(R_{n+m-1})$ = Row $n$ of $\Delta(R_{n+m-1})$.

\end{prop}

\begin{proof}
Since $R_k$ has $k+1$ terms, Lemma \ref{symmetry} shows:

\spa Column $m$ of $\Delta(R_k)$ = Row $k+1 - m$ of $\Delta(R_k)$.  

This proves the second equality in the statement of the Proposition. To complete the proof we  will show: \; $\fbin{n}{k}_{C_m} \;=\; \fbin{n}{k}_{R_{n+m - 1}}$ \ for every $n, k, m$.

Those sequences involved are:

\spa  $C_m(N) = \fbin{N+m-1}{m}_f$ \ and \ $R_{n+m-1}(N) = \fbin{n+m-1}{N-1}_f$. 
 Then

\medskip
\spb $\abin{n}{k}_{C_m} \!\!\!=
   \dfrac{ \phantom{.}\fbin{n+m - 1}{m}_f \fbin{n+ m-2}{m}_f\cdots \fbin{n+m-k}{m}_f }{\fbin{k+m - 1}{m}_f \fbin{k+m-2}{m}_f \;\; \cdots \;\; \fbin{m}{ \rule[-0.5mm]{0mm}{3mm}m}_f }$ \quad and 

\medskip
\spb $\abin{n}{k}_{R_{n+m-1}} \!\!\!\!= 
\dfrac{ \fbin{n+m-1}{n-1}_f \fbin{n+m-1}{n-2}_f \cdots \fbin{n+m-1}{n-k}_f }{ \fbin{n+m-1}{k-1}_f \fbin{n+m-1}{k-2}_f \cdots \fbin{n+m-1}{0}_f }$.

\medskip

The stated equality becomes a straightforward, but very long, calculation.  Substitute the $f$-factorial definitions for all those \bin\ coefficients and simplify the fractions.  We omit the many details.

\end{proof}

This formula shows that in the pyramid $\bp(f)$:  slices parallel to a face do yield the \bim\ triangles for the columns of $\Delta(f)$.

\begin{cor} \label{ColRowBin}
 If $f$ is a sequence, then:

\spa all columns of $\Delta(f)$  are \bin\ \; $\iff$ \;  all rows of $\Delta(f)$ are \bin.

Those conditions hold  when   $f$ is \bin\ at every level, as in Definition \ref{pyramid}.
\end{cor}

\bigskip
Note.  The 3-fold symmetry of $\Delta(R_m)$ (in Proposition \ref{symmetry}) implies that the formula in Proposition \ref{ColRow} is equivalent to:

\spc  $ \fbin{n}{\rule{0mm}{4mm}k}_{C_m} \;=\; \fbin{k + m}{m}_{R_{n+m - 1}}$, for every $n, k$, and $m$.

\bigskip

\newpage

\section{ Divisor-Product Sequences.}

\begin{dfn} \label{div_prod_def}
For a sequence $g$, define sequence $\pp(g)$ by: \      $\pp(g)(n): =  \prod\limits_{d|n} g(d).$ \\[-3pt]
Sequence $f$ is a \textit{divisor-product}  if $f = \pp(g)$ for an integer sequence $g$.   
\end{dfn}

\medskip
That product notation indicates that $d$ runs over the positive integer divisors of $n$.  
Cyclotomic polynomials provide motivation.
 
Define the homogeneous polynomials  $\Phi_n(x, y)$ in $\zz[x, y]$ by requiring: 

\hspace*{50mm}  $ \displaystyle \;\; x^n - y^n \;=\; \prod_{d \mid n} \Phi_d(x, y).$  \\[-5mm]

For example 
\begin{align}
\spa \Phi_1(x, y) &=\; x - y
      \spa & \Phi_4(x, y) &=\; x^2 + y^2 \notag \\
\spa \Phi_2(x, y) &=\; x + y
      \spa & \Phi_5(x, y) &=\;  x^4 + x^3y + x^2y^2 + xy^3 + y^4 \notag \\
\spa \Phi_3(x, y) &=\; x^2 + xy + y^2
      \spa & \Phi_6(x, y) &=\;  x^2 - xy + y^2   \notag 
\end{align}

Note that $\Phi_n(y, x) = \Phi_n(x, y)$ for every $n > 1$. 
Each (inhomogeneous) cyclotomic polynomial $\Phi_n(x) = \Phi_n(x, 1)$ is monic of degree 
$\varphi(n)$ with integer coefficients.\footnote{Further information appears in many number theory texts.  One convenient reference is \\ 
 \url{https://en.wikipedia.org/wiki/Cyclotomic_polynomial}. }

For example, the sequence $G_2 = (2^n - 1) = (1, 3, 7, 15, 31, 63, 127, 255, \dots)$ is a divisor-product since it factors as 

\hspace*{50mm}  $2^n - 1 = \prod\limits_{d|n} \Phi_d(2)$,

and  $\big( \Phi_n(2) \big) = (1, 3, 7, 5, 31, 3, 127, 17, 73, 11, \dots)$ has integer entries. \medskip

These sequences are ``divisible'' in the following sense.

\begin{dfn} \label{divi}
An integer sequence is \emph{divisible} if for $k, n \in \zz^+$:

\spb   $k \mid n$ \ implies \ $f(k) \mid f(n)$.

Such $f$ is called a \emph{divisibility sequence} or a \emph{division sequence}.

\end{dfn}

\begin{lem}  \label{DP_basics} \ 

\begin{enumerate}[label=(\arabic*),topsep=-2pt]  
\item If $f$ is a divisor-product then $f$ is divisible.

\item  Let $G_{a,b}(n) = \dfrac{\ a^n - b^n}{a - b} = a^{n-1} + a^{n-2}b + \cdots + ab^{n-2} + b^n$.  \\[3pt]
If $a, b$ are integers, then $G_{a,b}$ is a divisor-product.

\item  For $c \in \zz$, the  sequences  \\[5pt]
\spz  $I = (n) = (1, 2, 3, \dots)$, \; $(c) = (c, c, c, \dots )$, \;and\; $(c^n) = (c, c^2, c^3, \dots )$\\[5pt]
are divisor-products.  \\[-15pt]

\item \label{DP_prod} If $(a_n)$ and $(b_n)$ are divisor-products, then so is $(a_nb_n)$.

\item Suppose the integer sequence $\psi$ is homomorphic (as in Lemma \ref{bin-basic}).  \\
If $f$ is a divisor-product then so is $\psi \circ f$.

\end{enumerate}
\end{lem}

\begin{proof}

(1) follows from Definition \ref{div_prod_def}. Note also that if $f$ is a divisor-product and $d = \gcd(m,n)$, then:  $f(m)f(n) \mid f(mn)f(d)$.
\smallskip

(2) Check that $G_{a,b}(n) = \prod\limits_{d \mid n} g(d)$ where $g(1) = 1$ and $g(n) = \Phi_n(a, b)$ when $n > 1$.  This remains valid when $a = b$ since $G_{a, a}(n) = na^{n-1}$

(3)  Since $I = G_{1,1}$ we may apply (2).  Explicitly, $I =\pp(j)$  where \\
\spb $j(n) = 
\begin{cases}
   p  &\text{ if } n = p^k > 1 \text{ is a prime power}, \\
   1 & \text{ otherwise}.
\end{cases}$ \\

\spz $(c) = \pp(\delta_c)$ where $\delta_c = (c, 1, 1, 1, \dots)$.

\spz $(c^n) = \pp(c^{\varphi(n)})$, where $\varphi$ is the Euler function. 
\smallskip

(4) and (5): Check that $\pp(gh) = \pp(g)\pp(h)$, and that $\psi \circ \pp(g) = \pp(\psi \circ g)$.  \end{proof}

\medskip
For any never-zero sequence $f$ in $\qq$,  there exists a unique $g$ in $\qq$ with $f = \pp(g)$.   
The multiplicative version of the M\"{o}bius inversion formula provides an explicit formula for $\pp^{-1}(f)$. 

\medskip

\begin{lem}[M\"{o}bius inversion] \label{mobius} 
If $f$ is a never-zero sequence and $f = \pp(g)$ then: 

\spd  $g(n) =   \prod\limits_{d|n}f(d)^{\mu(n/d)}$.  
\end{lem}

The M\"{o}bius function $\mu(k)$ has values in $\{0, 1, -1 \}$.   By definition,  $f$ is a divisor-product exactly when every $g(n)$ is an integer. The additive form of M\"{o}bius inversion is discussed in many number theory texts.  This multiplicative form is a variant discussed explcitily in some textbooks.\footnote{One good reference is 
\url{https://en.wikipedia.org/wiki/Mobius_inversion_formula}.
}  We omit the proof of this well-known result.


\medskip
In 1939, Ward \cite{mW39} pointed out  that divisor-product sequences are  \bin.  The following generalization is the major result of this article.

\medskip
\begin{thm} \label{DP_Bin} A divisor-product sequence is \bin\ at every level. 
\end{thm}

\begin{proof}   
If $g$ is a never-zero integer sequence, we will prove that $\pp(g)$ is \bin\ at every level, that is:  Every entry of the pyramid $\bp\big( \pp(g) \big)$ is an integer.
To do this, consider a ``generic'' situation.

Let $\zz[X]$ be the polynomial ring built from an infinite sequence $X = (x_1, x_2, x_3, \dots)$ of independent indeterminates.  

\underline{Claim}. Each entry of $\bp\big( \pp(X) \big)$ is a polynomial in $\zz[X]$.  

If this Claim is true, we may substitute $g$ for $X$ to conclude that each term in $\bp\big( \pp(g) \big)$ is in $\zz$, proving the Theorem.
\smallskip

Each entry in the pyramid $\bp\big( \pp(X) \big)$ is a quotient involving products of terms $\fbin{n}{k}_{\pp(X)}$ for various $n$ and $k$.  Then it is a ``rational monomial,'' a quotient of monomials involving the variables $x_1, x_2, \dots\ $. 
\bigskip
For a real number $\alpha$, we write $\lfloor \alpha \rfloor$ for the ``floor function'': the greatest integer $\le \alpha$.  

Note that

\spa $\la n \ra_{\pp(X)} = \prod\limits _{j=1}^n \pp(X)(j) = (x_1)(x_1x_2)(x_1x_3)(x_1x_2x_4)(x_1x_5)(x_1x_2x_3x_6)\cdots  $

\hspace*{25mm} = $\displaystyle x_1^n x_2^{\lfloor n/2 \rfloor}x_3^{\lfloor n/3 \rfloor}\cdots x_n^{\lfloor n/n \rfloor} \ = \  \prod_{r>0} x_r^{\lfloor n/r \rfloor} $.  \hfill (a)

That infinite product (with $r = 1, 2, 3, 4, \dots$) is valid because $\lfloor n/r \rfloor = 0$ when $r > n$.  

For rational monomial $M$ and index $r$, write $v_r(M)$ for the exponent of $x_r$ in the factorization of $M$. That is:

\spd\spb  $\displaystyle M = \prod_{r>0} x_r^{v_r(M)}$. 

Then $v_r(M) \in \zz$, and $v_r(M) = 0$ for all but finitely many $r$. Moreover,  $M$ is a polynomial 
exactly when $v_r(M) \ge 0$ for every index $r$.

Formula (a) shows:\;   $v_r\big( \la n \ra_{\pp(X)} \big)  =  \lfloor n/r \rfloor$.  Therefore the exponent

\spc $\delta_{m,r}(j) :=  v_r\Big( \fbin{m+j}{m}_{\pp(X)} \Big)$ 

has the simpler formula:

\spc  $\displaystyle \delta_{m,r}(j) \;=\;  \left\lfloor \frac{m+j\rule{0mm}{3mm}}{r} \right\rfloor - \left\lfloor \frac{m}{r} \right\rfloor - \left\lfloor \frac{j}{r} \right\rfloor$.   

\smallskip
We will see below that this quantity is either 0 or 1.
\smallskip

To prove that $\pp(X)$ is \bin\ at every level, we need to show: For every $m$, the column sequence $C_m(n) = 
\fbin{m+ n -1}{m}_{\pp(X)}$ is \bin\ in $\zz[X]$.  In other words, for \linebreak 
every $n, a \in \zz^+$:

\spz $\fbin{m}{\rule{0mm}{3.9mm}m}_{\pp(X)}\cdot \fbin{m+1}{m}_{\pp(X)}\cdots \fbin{m+n-1}{m}_{\pp(X)}$ \\

\spb divides\; 
 $\fbin{a+m}{\rule{0mm}{3.5mm}m}_{\pp(X)}\cdot \fbin{a+m+1}{m}_{\pp(X)}\cdots \fbin{a+m+n-1}{m}_{\pp(X)}$ \; in $\zz[X]$. \\

The definition of $\delta_{m,r}(j)$ then shows:

\spz  $\pp(X)$ is \bin\ at level $m$ \ if and only if \ for every $n, a, r \in \zz^+$: \smallskip

\spd $\displaystyle \sum_{j=0}^{n-1} \delta_{m,r}(j) \;\le\; \sum_{j=a}^{a+n-1} \delta_{m,r}(j)$. \hfill (b)

\smallskip

The formula $\delta_{m,r}(j) \;=\; \left\lfloor \frac{m+j\rule{0mm}{3mm}}{r} \right\rfloor - \left\lfloor \frac{m}{r} \right\rfloor - \left\lfloor \frac{j}{r} \right\rfloor$,  implies \ $\delta_{m+rs,r}(j) = \delta_{m,r}(j)$ for any $s \in \zz$.  Then we may alter $m$ to assume $0 \le m < r$.  

Similarly, $\delta_{m,r}(j+rs) = \delta_{m,r}(j)$. Then for fixed $r, m$, the value  $\delta_{m,r}(j)$ depends only on  ($j$ mod $r$).  Consequently, any block of $r$ consecutive terms in the sums in (b) yields the same answer, namely the sum over all values in $\zz/r\zz$.  Then if $n \ge r$ we may cancel the top $r$ terms of each sum in (b) and replace $n$ by $n - r$.  By repeating this process, we may assume $0 \le n < r$.

For real numbers $\alpha, \beta \in [0, 1)$, a quick check shows that

\spc  $\lfloor \alpha + \beta \rfloor - \lfloor \alpha \rfloor - \lfloor \beta \rfloor \;=\;  \begin{cases}
       1  \ \text{ if }   \alpha + \beta \ge 1,  \\
       0  \ \text{ if }      \alpha + \beta < 1.
    \end{cases}$

In our situation $0 \le m, n < r$ and we may write $j \equiv j' \pmod{r}$ where $0 \le j' < r$.   Then \ $\delta_{m,r}(j) \;=\;  \begin{cases}
       1  \ \text{ if }   m + j' \ge r,  \\
       0  \ \text{ if }  m + j' < r.
    \end{cases}$ \ 
This says that the sequence $\delta_{m,r}$ begins with $r - m$ zeros followed by $m$ ones, and that pattern repeats with period $r$.  For instance, when $m = 2$ and $r = 6$ the sequence is

\spc $\delta_{2, 6} = (0, 0, 0, 0, 1, 1, 0, 0, 0, 0, 1, 1, \ \dots)$.

For such a sequence, it is not hard to see that for the sums of any ``window'' of $n$ consecutive terms, the minimal value is provided by the $n$ initial terms.  This proves the inequality in (b) and completes the proof of the Theorem.
\end{proof}

\bigskip
Further examples of divisor-products are provided by sequences that satisfy a linear recurrence of degree 2.   The Fibonacci sequence and $(2^n - 1)$ are included in this class of ``Lucas sequences.''

\begin{lem}\! \label{lucas}
For $P, Q \in \zz$ not both zero, define the \emph{Lucas sequence} $\lu = \lu_{P, Q}$ by setting $\lu(0) = 0$ and $\lu(1) = 1$, and requiring

\spc $\lu(n+2) = P\sdot \lu(n+1) - Q\sdot\lu(n)$ \; for every $n \ge 0$.  

Then  $\lu$ is a divisor-product.

\end{lem}

\begin{proof} 
Factor $x^2 - Px + Q = (x - \alpha)(x -\beta)$ for $\alpha, \beta \in \cc$.  Then $\alpha, \beta$ are not both zero.  Recall the well-known formulas:  

\spb If $\alpha \ne \beta$ then:  $U(n) = \frac{\alpha^n - \beta^n}{\!\!\alpha\,-\,\beta}$ \; for every $n \in \zz^+$,\\[8pt]
\spb If $\alpha = \beta$ then:  $U(n) = n\alpha^{n-1}$ \; for every $n \in \zz^+$.\\[8pt]
To verify those formulas, note that the sequences $(\alpha^n)$ and $(\beta^n)$ satisfy that recurrence.  Then so does every linear combination $( c\alpha^n + d\beta^n )$. When $\alpha \ne \beta$, the stated formula has this form and matches $U(n)$ for $n = 0, 1$. Induction shows that those quantities are equal for every $n$.

When $\alpha = \beta$ check that $\alpha \in \zz$ and apply same method to prove that $\lu(n) = n\alpha^{n-1}$.  In this case, $\lu$ is a divisor-product since Lemma \ref{DP_basics} implies that $(n)$, $(\alpha^{n-1})$, and their product are divisor-products.  

Suppose $\alpha \ne \beta$.  Since $\alpha^n - \beta^n = \prod\limits_{d|n}\Phi_d(\alpha, \beta)$ and $\Phi_1(\alpha, \beta) = \alpha - \beta$, then $\lu = \pp(g)$ where $g(1) = 1$ and $g(n) = \Phi_n(\alpha, \beta)$ when $n > 1$.  Then $\lu$ is a divisor-product provided $g$ has integer values.  M\"{o}bius (Lemma \ref{mobius}) implies that every $g(n)$ is a rational number.  Since $\alpha, \beta$ are algebraic integers and each $\Phi_n$ has integer coefficients, we know that $g(n)$ is an algebraic integer.  Therefore each $g(n) \in \zz$.\footnote{To avoid theorems about algebraic integers, we may use the theory of symmetric polynomials: 
If $p \in \zz[x, y]$ and $p(x, y) = p(y, x)$, then $p \in \zz[\sigma_1, \sigma_2]$, where $\sigma_1 = x + y$ and $\sigma_2 = xy$  are the elementary symmetric polynomials.  Note that $\sigma_1(\alpha, \beta) = \alpha + \beta = P$ and  $\sigma_2(\alpha, \beta) = \alpha\beta = Q$. }  
\end{proof}

\medskip

For the Fibonacci sequence $F = \lu_{1, -1}$, this Lemma implies that $F = \pp(b)$ where 
$b = (1, 1, 2, 3, 5, 4, 13, 7, 17, 11, 89, 6,  \dots)$ is an integer sequence.

\bigskip

We end this section with a few more remarks about divisor-products.

The triangular number sequence $T$ is \bin\ but not a divisor-product.  ($T$ is not even divisible.)
The sequence $w := (1, c, c, c, c, \dots)$
 is a divisor-chain, so it is  \bin\ at every level by Lemma \ref{every_level_1}.  But when $c > 1$ it is not a divisor-product since $w_2w_3$ does not divide $w_1w_6$. 

\medskip
Divisibility does not imply \bin.  
Sequence $h(n) := 
\begin{cases}
    1 &\text{ if } n = 1, 5, 7, \\
    c &\text{ otherwise}
\end{cases}$  is divisible but $h$ is not \bin\ when $c > 1$.  (For $h_1h_2h_3$  does not divide $h_5h_6h_7$.) 

\bigskip
It is curious that the Euler function $\varphi(n)$ is a divisor-product. 
To see this, recall that $\varphi$ is \emph{multiplicative:} \ $\varphi(ab) = \varphi(a)\varphi(b)$ whenever $a, b$ are coprime. Standard formulas imply that $\varphi$ is divisible as in Definition \ref{divi}.  
Then the next result applies to $\varphi$.
\medskip

\begin{prop}
A 
multiplicative divisibility sequence is a divisor-product.
\end{prop}

\begin{proof}
If $f$ is a never-zero sequence, let $f = \pp(g)$ for a sequence $g$ in $\qq$.  If $f$ is
multiplicative, then \ 
\; $g(n) =
 \begin{cases} 
   \frac{ f(p^m) }{ f(p^{m-1}) }  &\text{ if } n = p^m > 1 \text{ is a prime power}, \\
   \quad 1 \rule{0mm}{5mm}                    &\text{ \ otherwise}.
\end{cases}$  

If $f$ is also divisible, every $g(n)$ is an integer and $f$ is a divisor-product. 
\end{proof}

\vfill
\pagebreak
\begin{aside}

 Suppose $f = \pp(g)$ is a divisor-product with $f(1) = 1$.  Then: 
\begin{enumerate}[label=(\arabic*),topsep = 0pt]  
\item $f$ is multiplicative \;if and only if\; $g(n) = 1$ whenever $n$ is not a prime power.

\item $f$ is homomorphic  \;if and only if\; $g(n) = 1$ whenever $n$ is not a prime power and $g(p^m) = g(p)$ for every prime power $p^m > 1$.

\item $f$ is a GCD sequence (see Definition \ref{gcd_def} below) \;if and only if\; 
$g(m)$ and $g(n)$ are coprime whenever $m \nmid n$ and $n \nmid m$. 
\end{enumerate}
\end{aside}

Property (3) was pointed out in \cite{DB08}.
\vfill

\newpage

\section{ Related Topics.}

\subsection{GCD Sequences.}

Let's first introduce notations motivated by lattice theory: 

\spc  $a \wedge b := \gcd(a, b)$ \; and \; $a \vee b := \lcm(a, b)$.

Those operations are defined on $\zz$, except that $0 \wedge 0$ is undefined.

\medskip
\begin{dfn} \label{gcd_def}
An integer sequence $f$ is a \textit{\gc\ sequence} if:

\spb  $f(m \wedge n) = f(m) \wedge f(n)$ \; for every $m, n \in \zz^+$

Using earlier notation, this says that for every $m, n$ with  $d = \gcd(m,n)$:\\
\spd   $\gcd(f_m, f_n ) = f_d$.
\end{dfn}

 Other authors use different names for sequences with this property. Hall \cite{mH36} and Kimberling \cite{cK79} calls them \emph{strong divisibility} sequences. Granville \cite{aG22} considers sequences that satisfy a linear recurrence, and refers to those with this GCD property  as \emph{strong LDS's} (linear division sequences).   Knuth and Wilf \cite{KW89} use the term  \emph{regularly divisible}.  Dziemia\'n{}czuk and Bajguz \cite{DB08} call them \emph{GCD-morphic} sequences.

Carmichael \cite{rC13} proved that every Lucas sequence (as in Lemma \ref{lucas}) is GCD. In 1936 Ward \cite{mW36} used prime factorizations to prove that every GCD sequence is \bin, a result also proved in \cite{KW89}.   The following Theorem together with Theorem  \ref{DP_Bin} implies: Every GCD sequence is \bin\ at every level.

\begin{thm} \label{GCD-divprod} Every \gc\ sequence is a divisor-product.
\end{thm}

This result has been  proved independently by several authors so we do not include a proof here. Kimberling \cite{cK79} showed that M\"{o}bius Inversion (see Lemma \ref{mobius}) applied to a GCD sequence produces integers.  Dziemia\'n{}czuk and Bajguz \cite{DB08} gave a quite different proof. A separate argument can be given by first reducing to the case of sequences of type $f(n) = c^{a(n)}$.

Note.  If $f, g$ are GCD sequences, then so are $f \wedge g$ and $f \circ g$.  For instance, $2^{F_n} - 1$ is a GCD sequence, so it is \bin\ at every level. This example was mentioned in \cite{AK74}.

\subsection{ComboSum Sequences}

An inductive proof that all binomial coefficients $\binom{n}{k}$ are integers uses the formula

\spd\spz  $ \binom{\,n+1\,}{\rule{0ex}{2ex}k} \;=\; \binom{\,n\,}{\rule{0ex}{2ex}k} + \binom{\,n\,}{\rule{0ex}{2ex}k-1}. $

We extend that recurrence formula to our context. Other authors have pointed out some of these ideas.  See \cite{vH67}, \cite{LS10}, \cite{mD14}.

\begin{lem} \label{bin_recurse}
Suppose $f$ is a sequence, $1 \le k \le n$, and $\la n \ra_f \ne 0$. 

 If $f_{n+1} = uf_{n-k+1} + vf_k$ for some  $u, v$, then: \smallskip

\spd\spz $\fbin{\rule{0ex}{0.8ex}n+1}{\rule{0ex}{2ex}k}_f =\; u\!\cdot\!\fbin{\rule{0ex}{0.8ex}n}{\rule{0ex}{2ex}k}_f +\ v\!\cdot\!\fbin{\rule{0ex}{0.8ex}n}{\rule{0ex}{2ex}k-1}_f$. 
\end{lem}

\textit{Proof.} Given  $\frac{f_{n+1}}{f_kf_{n-k+1} } = u\!\cdot\!\frac{1}{f_k} + v\!\cdot\!\frac{1}{f_{n-k+1} }$, multiply by $\frac{\la n \ra_f }{\la k-1 \ra_f \la n-k \ra_f }$ to find \smallskip

\spb $\frac{\la n+1 \ra_f }{\la k \ra_f \la n-k+1 \ra_f } \;=\; u\!\cdot\!\frac{\la n \ra_f }{\la k \ra_f \la n-k \ra_f } \;+\; v\!\cdot\!\frac{\la n\ra_f }{\la k-1 \ra_f \la n-k+1 \ra_f }$.  

This is the stated conclusion. \qed

\smallskip

\begin{dfn} \label{divGCD} An integer sequence $f$ is a \emph{\dgcd} sequence if for every $m, n$: 

\spc  $f_m \wedge f_n  \mid f_{m+n}$.

\end{dfn}

By elementary number theory, that condition is equivalent to saying:\\
\spc $f_{m+n} = u\sdot f_m + v\sdot f_n$ \ for some $u, v \in \zz$. 

Certainly any GCD sequence is a \dgcd.  In particular, the Lucas sequences $\lu_{P,Q}$ have the \dgcd\ property.  

\medskip
\begin{prop}\label{gcd-criterion}
 Every \dgcd\ sequence is \bin.  
\end{prop}

\begin{proof}  If $f$ is \dgcd\ and $0 < k < n$,  then
 \ $f_{n+1} = uf_{n-k+1} + vf_k$,  for some integers $u, v$,
and Lemma \ref{bin_recurse} applies.  An inductive proof shows that  $f$ is \bin.    
\end{proof}

\bigskip
Every \dgcd\ sequence is divisible.  For when $m = n$, Definition \ref{divGCD} implies: $f(n) \mid f(2n)$. An inductive argument shows: $f(n) \mid f(kn)$ for every $k \in \zz^+$.

\medskip

Sequence $(1, 2, 2, 2, \dots)$ is a \dgcd\ sequence that is not a divisor-product.  

I don't know whether every \dgcd\ sequence is \bin\ at every level.


\subsection{Polynomial Sequences}

Which polynomials satisfy our various conditions on integer sequences? 

A polynomial $f \in \cc[x]$ is called \textit{integer-valued} if $f(n) \in \zz$ for every $n = 1, 2, 3, \dots$  

\medskip
\begin{prop}
 Suppose $f$ is an integer-valued polynomial of degree $d$.  \\
If $f(n) \mid f(2n)$ for infinitely many $n \in \zz^+$, then $f(x) = bx^d$ for some $b \in \zz$.
\end{prop}

The proof is omitted here. This result determines all polynomial sequences that are divisible.  It is more difficult to determine which polynomials are \bin..  

Recall that each Pascal column polynomial 

\spd $C_m(x) = \binom{x+m-1}{m}$ 

is integer-valued, degree $m$, and $C_m(1) = 1$.  By Theorem \ref{DP_Bin}, $C_m$ is \bin.  

\medskip
For more examples, note that the polynomial

\spd $H_m(x) = \binom{mx}{m}$ 

is integer-valued, degree $m$, and $H_m(1) = 1$. 
Check that $\la n \ra_{H_m} = \frac{ (mn)! }{\ (m!)^n }$, so that: $\fbin{n}{k}_{H_m}  \;=\;   \binom{mn}{\rule{0mm}{3mm}mk}$. Since binomial coefficients are integers, $H_m$ is \bin.
\medskip

\begin{thm} Let $f$ be a \bin\ polynomial sequence with $f(1) = 1$.  If $\deg(f) \le 2$  then $f$ is one of:
 \spz $1, \spz x, \spz \binom{x+1}{2}, \spz x^2, \spz  \binom{2x}{2}.$ 
\end{thm}

\smallskip

The proof involves many details and will appear in a separate paper. Listing all the degree 3 \bin\ polynomials seems to be much more difficult.

\subsection{Linear Recurrences}

Suppose   $f$ is an integer sequence satisfying a linear recurrence of order 2:

\spc $f(n+2) = P\sdot f(n+1) - Q\sdot f(n)$  \; for every $n \ge 1$, \hfill (d)

where $P, Q \in \zz$. Suppose the associated polynomial.is:

\spb $p(x) = x^2 - Px + Q = (x - \alpha)(x - \beta)$, \;  for $\alpha, \beta \in \cc$. 

If $f(0) = 0$ then $f$ is a constant multiple of the Lucas sequence $\lu_{P,Q}$ of Lemma \ref{lucas}, and $f$  enjoys most of the properties mentioned above:  it is a divisibility sequence, a GCD sequence, a divisor-product, a \dgcd, and is \bin\ at every level.  

\spb If $f(0) \ne 0$, can $f$ still satisfy some of those properties?

If $Q = 0$ then $f$ satisfies a recurrence of order 1:  $f(n+1) = Pf(n)$ for $n \ge 2$.  
Then $f(n)$ has the form $a\sdot P^{n-1}$ (for $n \ge 2$) and it's not hard to determine which of those properties $f$ satisfies.  We assume below that $Q \ne 0$.

For sequences $f$ satisfying a linear recurrence (of any order), Kimberling \cite{cK78} proved that if $f$ is a GCD sequence with $f(0) \ne 0$, then $f$ must be periodic.  For the order 2 case, all the periodic GCD sequences are listed in \cite{HS85}.  The next result involves a much weaker hypothesis.
\medskip

\begin{prop} \label{degenerate}
Suppose $f$ satisfies linear recurrence (d) above, and $Q\sdot f(0) \ne 0$.\linebreak
  If $f$ is divisible, then $\alpha/\beta$ is a root of unity.
\end{prop}

\textit{Proof Outline.}
If a divisibility sequence $f$ satisfies a linear recurrence, Hall \cite{mH36} noted that every prime factor of any $f(n)$ also divides $Q\sdot f(0)$. (In fact, for the order 2 case: $f(n) \mid Q^n\sdot f(0)$ for every $n$.) Then the set of all $f(n)$ involves only finitely many different prime factors. 

Ward \cite{mW54} showed: If $f$ is non-degenerate (meaning that $\alpha/\beta$ is not a root of unity), then the values $f(n)$ involve infinitely many primes.  This completes the proof.  \qed
 
 \medskip
A version of this Proposition is valid for all linearly recurrent sequences, not just those of order 2.  The proof uses \cite{rL74}, where Laxton generalized Ward's Theorem.

\medskip
Proposition \ref{degenerate} can be used to make a complete list of divisibility sequences that satisfy a recurrence of order 2.  In addition to the Lucas sequences and exponential sequences, there are a few periodic cases, with periods 1, 2, 3, 4 or 6. 

\medskip
A. Granville \cite{aG22} has studied dvisibility sequences in much greater depth.

\bigskip

In summary, among all sequences $f$ satisfying recurrence (d), we can list all those that are \gc, or divisor-product, or \dgcd\  (since each of those properties implies divisibility).  The situation is more difficult for \bin\ sequences.

\begin{oques} \ 
\begin{enumerate}[label=(\arabic*), itemsep=3pt, topsep=0pt]

\item Which sequences $f$ satisfying recurrence (d) are \bin? If such $f$ is \bin\ (or \bin\ at every level), must $f$ be a divisibility sequence?  
\item What if we allow linearly recurrent sequences of order $> 2$\,? 
\end{enumerate}
\end{oques}

For (2), note that sequence $T(n) = \binom{n}{2}$ is \bin\  and $T$ satisfies a linear recurrence of order 3 with polynomial $p(x) = (x - 1)^3$. However $T$ is not divisible and is not \bin\ at level 2.

\bigskip\bigskip

{\bf Acknowledgments.} It is a pleasure to thank Jim Fowler, Paul Pollack, Zev Rosengarten, and the anonymous referee for their helpful comments and suggestions.

\end{document}